\documentclass[11pt]{amsart}
\usepackage{amssymb, amscd, amsmath, amsthm, latexsym, hyperref,xcolor}
\usepackage{pb-diagram, fancyhdr, graphicx, psfrag}

\newcommand{\compactlist}{\begin{list}{$\bullet$}{\setlength{\leftmargin}{1em}}}

\DeclareMathOperator{\cfk}{\it CFK}

\newcommand{\spinc}{\ifmmode{{\mathfrak s}}\else{${\mathfrak s}$\ }\fi}
\newcommand{\spinct}{\ifmmode{{\mathfrak t}}\else{${\mathfrak t}$\ }\fi}
\newcommand{\spincw}{\ifmmode{{\mathfrak w}}\else{${\mathfrak w}$\ }\fi}
\def\Z{\mathbb Z}

\def\R{\mathbb R}

\def\Spc{Spin$^c$}

\newtheorem{theorem}{Theorem}[section]
\newtheorem{question}[theorem]{Question}
\newtheorem{thm}[theorem]{Theorem}

\newtheorem*{theorem*}{Theorem}
\newtheorem{lemma}[theorem]{Lemma}
\newtheorem{corollary}[theorem]{Corollary}
\newtheorem{cor}[theorem]{Corollary}

\newtheorem{proposition}[theorem]{Proposition}
\theoremstyle{definition}
\newtheorem{definition}[theorem]{Definition}
\newtheorem{example}[theorem]{Example}
\theoremstyle{remark}
\newtheorem*{ack}{Acknowledgements}
\newtheorem{remark}[theorem]{Remark}

\numberwithin{equation}{section}

\begin{document}
\title[The $\Upsilon$ function of $L$--space knots]{The $\Upsilon$ function of $L$--space knots is a Legendre transform}

\author{Maciej Borodzik}
\address{Institute of Mathematics, University of Warsaw, ul. Banacha 2,
02-097 Warsaw, Poland}
\email{mcboro@mimuw.edu.pl}

\author{Matthew Hedden} 
\address{Matthew Hedden: Department of Mathematics, Michigan State University, East Lansing, MI 48824 }
\email{mhedden@math.msu.edu}

\thanks{ The first author was supported by  Polish OPUS grant No 2012/05/B/ST1/03195}
\thanks{The second author was supported by NSF CAREER grant DMS-1150872, and an Alfred P. Sloan Research Fellowship}

 \thanks{\today}


\begin{abstract}
Given an $L$--space knot  we show that its $\Upsilon$ function  is 
the Legendre transform of a counting  function equivalent to the $d$--invariants of its large surgeries. 
The unknotting obstruction obtained for the $\Upsilon$ function is,
in the case of $L$--space knots, contained in the $d$--invariants of large surgeries.
Generalizations apply for connected sums of $L$--space knots, which imply that the slice obstruction provided by $\Upsilon$ on the subgroup of concordance generated by $L$--space knots is no finer than that provided by the $d$--invariants.

\end{abstract}

\maketitle

\section{Introduction}

In this note we compare two useful invariants of the smooth concordance group coming from Heegaard Floer homology.  The first is the   $\Upsilon(t)$ function recently defined by Ozsv\'ath--Stipsicz--Szab\'o  \cite{OSS}, and the other  is the set of  $d$--invariants of  $3$-manifolds obtained by large surgery on a knot. The latter can be encoded in a function, denoted $J(x)$, determined by the knot Floer homology invariants.  Our main observation is the following result:

\begin{thm}\label{thm:legendre}
For  $L$--space knots $\Upsilon(t)$  is the Legendre transform of  the function $x\mapsto 2J(-x)$.
\end{thm}
\noindent Here and throughout, an $L$--space knot is a knot on which {\em positive} framed surgery yields an $L$--space.  The result extends to connected sums of $L$--space knots, yielding the following corollary:
\begin{cor}\label{cor:independence} Let $\mathcal{L}$ denote the subgroup of the smooth concordance group generated by $L$--space knots.  Suppose $\Upsilon_\alpha(t)\ne 0$ for some $\alpha\in \mathcal{L}$.  Then the $d$--invariants of surgeries can be used to show $\alpha\ne 0$.
\end{cor}

The subgroup $\mathcal{L}$ is quite interesting.  In particular, it contains the subgroup $\mathcal{A}$ generated by algebraic knots, i.e. connected links of complex singularities, which lies at the crossroads of many interesting areas of mathematics \cite{Milnor,EN,KM,Rudolph,HKL,MicallefWhite}.  In this context, the $J$ function arises naturally as a counting function associated to the semigroup of the singularity defining an algebraic knot.  It is conjectured that $\mathcal{A}$, and $\mathcal{L}$ more generally, is freely generated \cite{Rudolph,HKL,Baker} and  it would be interesting to know whether the $d$--invariants provide strictly more information about this conjecture than $\Upsilon$ (it is known that neither can solve it, for example the two knots $T(2,13)\#T(2,3;2,15)$ and $T(2,15)\#T(2,3;2,13)$ have
the same $J$ functions and the same $\Upsilon$ functions, but in \cite{HKL} it is shown that they are not concordant).

In a related direction, one can compare criteria derived from $\Upsilon$ and $J$ for estimating the Gordian distance between knots.  For Gordian distance between $L$--space knots, the $d$--invariants do indeed contain more information.  

\begin{thm} \label{crossingstrength} Suppose $K_0$ and $K_1$ are connected sums of  $L$--space knots related by a sequence of crossing changes.  Then the crossing change inequality satisfied by their $\Upsilon$ functions   is determined (under the Legendre transform) by a crossing change inequality for their $J$ functions, but not conversely.
\end{thm}
\noindent See \eqref{eq:u2} and \eqref{eq:u1} below for the precise statement of the crossing change inequalities.  We highlight our interest in the above theorem by noting again that  algebraic knots, and  torus knots in particular, are $L$--space knots.   
The minimal unknotting sequences of torus knots have recently attracted a lot of  interest; see for example
\cite{Ba-un,Ba0,Fel15,OwSt,Wn}. The Gordian
distance between algebraic knots
is closely related to
studying adjacency of singularities; see \cite{bo-li2,Fel14}.

It is important to note that Theorem \ref{thm:legendre} does not extend to all knots.  Indeed, the Legendre transform of a real-valued function is always convex, whereas $\Upsilon(t)$ is typically not.  For instance, the mirror image of an $L$--space knot will have concave $\Upsilon(t)$ function, since  taking mirror images  changes the sign of $\Upsilon(t)$.  On the other hand, its $J$--function will be exactly the same as that of the unknot, with Legendre transform  identically zero.  It is then natural to ask in what capacity Theorem~\ref{thm:legendre} and its corollaries extend. 
\begin{question} For which knots is $\Upsilon(t)$ a convex function?  For which such knots is $\Upsilon(t)$ the Legendre transform of $2J(-x)$?
\end{question}

From a geometric perspective, a natural extension of the set of $L$--space knots are the  so-called {\em strongly quasipositive knots}, distinguished by the fact that they possess a minimal genus Seifert surface properly isotopic to a piece of an algebraic curve in the $4$--ball.  Fibered strongly quasipositive knots are detected by their knot Floer homology \cite{SQPfiber}, and it would be very interesting to know if $\Upsilon$ provides further information about this feature.   For instance:   

\begin{question}\label{question} Suppose $K$ is a strongly quasipositive knot or, more generally, a quasipositive knot.  Is $\Upsilon_K(t)$ convex?\end{question}
\begin{remark}
Peter Feller and David Krcatovich, and independently Jen Hom,  informed us of examples which indicate that the answer to the above question is no.  Feller and Krcatovich's examples come from closures of the family of $3$--braids $(\sigma_1\sigma_2^2\sigma_1)^n\sigma_1\sigma_2$, for $n\ge 3$, and Hom's example is the $(2,1)$ cable of the right-handed trefoil (see \cite{FK} for an explanation of the former examples and \cite{Hom,HedThesis,Petkova} for calculations of the knot Floer homology of the latter, from which $\Upsilon$ can be readily extracted).   
\end{remark}

Finally, it would be interesting to know if there is some generalization of Theorem \ref{thm:legendre} which  holds for all knots.  Such a generalization would likely incorporate the $d$--invariant counting function for negative-framed surgeries.

\begin{ack}  This note grew out of joint work with Chuck Livingston, and benefitted from his input.  We also thank Peter Feller and Marco Golla for interesting conversations, and Peter Feller, David Krcatovich, and Jen Hom for answering Question \ref{question}.
\end{ack}

\section{Review of the Legendre transform}
We give some necessary background on Legendre transform of functions in one variable.
We refer to \cite[Section 12]{Rock} for more details. 

\begin{definition}\label{def:fl}
Let $f\colon\R\to\R$ be a continuous function. The \emph{Legendre transform} of $f$ is a function $f^*\colon\R\to\R\cup\{\infty\}$ defined as
\[f^*(t)=\sup_{x\in\R} tx-f(x).\]
The \emph{domain} of $f^*$ is the set $D(f^*)=\{t\colon f^*(t)<\infty\}$.
\end{definition}

\goodbreak
\begin{remark}\
\begin{itemize}
\item[(a)] Although in many articles the Legendre transform is defined only for convex functions, Definition~\ref{def:fl} does not require
$f$ to be convex.
\item[(b)] The Legendre transform is also known as the \emph{Fenchel-Legendre transform} or the \emph{convex conjugate}.
\item[(c)] One can consider the \emph{concave conjugate} by replacing the supremum in Definition \ref{def:fl} with infimum.  This would likely be relevant to any generalization of the results of this note to arbitrary knots.
\end{itemize}
\end{remark}

\begin{lemma}\label{lem:isconvex}
The function $f^*$  is a convex function.
\end{lemma}
\begin{proof}
For fixed $x$, the function $t\mapsto tx-f(x)$ is a convex function. A supremum of a family of convex functions is convex.
\end{proof}
Notice that this implies that $f^*$ is a continuous function on $D(f^*)$. 
\begin{remark}
If $f$ is a strictly convex function, then $(f^{*})^*=f$; see \cite[Theorem 12.2]{Rock}. This is not always true, for example, if $f$ is not convex, 
then $(f^*)^*$ is a convex function, so cannot be equal to $f$.
\end{remark}

For any function $h\colon\R\to\R$ and a number $y$ we define the shifted function $T_yh\colon\R\to\R$  by the formula
\begin{equation}\label{eq:shiftoperator}
T_yh(x)=h(x+y).
\end{equation}
We have:  
\begin{equation}
(T_yh)^*(t)=h^*-yt\label{eq:tyh}.
\end{equation}
To prove this,  write
\[(T_yh)^*(t)=\sup_{x\in\R} tx-h(x+y)\stackrel{z=x+y}{=}\sup_{z\in\R} (z-y)t-h(z)=h^*(t)-yt.\]

We will also need the following easy result.
\begin{lemma}\label{lem:fenchelineq}
Suppose $f$ and $h$ are two continuous functions satisfying $f(x)\le h(x)$ for all $x\in\R$. Then $f^*(t)\ge h^*(t)$ for all $t\in\R$.
\end{lemma}

Suppose now that $f$ and $g$ are two functions bounded from below. Define the \emph{infimum convolution} as
\begin{equation}\label{eq:diamond} f\diamond g(m)=\inf_{i+j=m} f(i)+g(j).\end{equation}
The above definition works for functions on any group $G$. We will use it over $\R$ or $\Z$. The following fact
relates the convolution to the Legendre transform.
\begin{lemma}\label{lem:diamond}
For two functions $f$ and $g$ we have
\[(f\diamond g)^*(t)=f^*(t)+g^*(t),\]
for all $t$ such that both sides are defined.
\end{lemma}
\begin{proof}
We have
\begin{multline*}
(f\diamond g)^*(t)=\sup_{x} tx-(f\diamond g)(x)=\sup_x tx-(\!\!\!\inf_{u+v=x} f(u)+g(v))=\\
=\sup_x\sup_{u+v=x} tx-f(u)-g(v)=\sup_u\sup_v tu+tv-f(u)-g(v)=f^*(t)+g^*(t).
\end{multline*}
\end{proof}

\section{The  $\Upsilon$ function for a knot $K$}
To a  knot $K$ in the $3$-sphere, knot Floer homology associates  a $\Z\oplus\Z$-- filtered, $\Z$--graded complex over $\mathbb{Z}_2$, denoted $\cfk^\infty(K)$, well-defined up to $\Z\oplus\Z$--filtered chain homotopy equivalence \cite{OSknots,Rasmussen}.  It is a module over $\Z_2[U,U^{-1}]$,
where $U$ is a formal variable whose action lowers the grading by $2$ and the filtration by $(1,1)$.  In \cite{OSS} (see also \cite{Liv}), this complex was used to define a function $\Upsilon_K\colon[0,2]\to\R$ associated to $K$. It is a generalization of the $\tau$ invariant, in the sense that
$\Upsilon'(0)=-\tau$.  Here we summarize its main properties

\begin{theorem}[see \expandafter{\cite[Proposition 1.8, Proposition 1.10, Theorem 1.11]{OSS}}]\label{thm:ossmain}
For $t\in[0,2]$, $\Upsilon(t)$ is a continuous, piecewise linear, function with the following properties: 
\begin{itemize}
\item[(a)] {\em (Symmetry)} $\Upsilon(t)=\Upsilon(2-t)$.
\item[(b)] {\em (Additivity)} If $K_1\#K_2$ denotes the connected sum of   knots $K_1$ and $K_2$, then $$\Upsilon_{K_1\# K_2}=\Upsilon_{K_1}+\Upsilon_{K_2}.$$
\item[(c)] {\em(Crossing Change Inequality) }If $K_-$ is obtained from $K_+$ by changing a positive crossing,
then \begin{equation*}\Upsilon_{K_+}(t)\le \Upsilon_{K_-}(t)\le \Upsilon_{K_+}(t)+t.\end{equation*}
\item[(d)]{\em (Slice genus bound) }If $g_s$ denotes the smooth slice genus, then for any $t$ we have $$|\Upsilon_K(t)|\le tg_s(K).$$\item[(e)] {\em(Mirror Reversal)} If $-K$ denotes the mirror image of $K$, with string orientation reversed, then $$\Upsilon_{-K}=-\Upsilon_{K}.$$
\end{itemize}
\end{theorem}
\noindent Note that (b) and (d) together imply that  $\Upsilon$ descends to a homomorphism from the smooth concordance group to the additive group of continuous real valued functions on the interval $[0,2]$.  Also note that (e) is implied by (b) and (d), since $-K$ is the concordance inverse of $K$.

\section{The $J$--function for an $L$--space knot}

Suppose $K$ is an $L$--space knot.
By \cite{os-lspace} the Alexander polynomial of $K$ is of the following form:

\begin{equation}\label{eq:deltaLspace}
\Delta_K(t)=\sum_{k=0}^{2n} (-1)^k t^{\alpha_k},
\end{equation}
for some decreasing sequence of integers $\alpha_0,\ldots,\alpha_{2n}$, where $\alpha_0=-\alpha_{2n}=g$ is the genus of $K$. 
Moreover, for an $L$--space knot the Alexander polynomial determines  $\cfk^\infty(K)$ complex which, in turn,
 determines the $d$--invariants of surgeries on $K$ \cite{OSknots,NiWu}.  This procedure is described in detail in \cite{bo-li}. Namely, write \eqref{eq:deltaLspace} in the following
form:
\[\Delta_K(t)=t^{-g}\left(1+(t-1)\left(t^{\beta_1}+t^{\beta_2}+\ldots+t^{\beta_{s}}\right)\right).\]
The numbers $\beta_1,\ldots,\beta_{s}$ are positive integers, which can be expressed in terms of the $\alpha$ coefficients. Consider the set
\[\mathcal{G}=\Z_{<0}\cup\{\beta_1,\ldots,\beta_{s}\}\]
and define
\[I(m)=\#\{x\in\Z\colon x\ge m,\ x\in\mathcal{G}\}.\]

\noindent We call $I(m)$  the \emph{gap function} for the knot $K$.  If $K$ is an algebraic knot, then $\Z\setminus\mathcal{G}$ is the semigroup of the corresponding singular point; see \cite[Chapter 4]{Wa} for
details. This motivates the terminology: $\mathcal{G}\setminus \Z_{<0}$ is the set of `gaps' in the semigroup of the singularity i.e. the elements of $\Z_{\ge 0}$ not included in the semigroup.  The gap function counts the number of such elements greater than or equal to a fixed integer.

\begin{example}\label{example:gaps}
The torus knot $T(6,7)$ is the link of the singularity at the origin of the curve $z^6+w^7=0$, which has semigroup generated by $6$ and $7$.   The corresponding gap set is
\[\mathcal{G}_{6,7}=\Z_{<0}\cup \{1,2,3,4,5,8,9,10,11,15,16,17,22,23,29\}.\]
Some sample values of the gap function are given below:
\smallskip
\begin{center}  \begin{tabular} {| c || c | c | c | c| c |c | c | c |c|c|c|c|c|c|c|}
   \hline
   $m$ &   $\ge 30$ & $29$ & $28$ & $23$ & $22$ &$21$ & $17$ & $16$ & $15$ &  $1$& $0$ & $-1$ & $-2$ & ...\\ \hline
    $I_{6,7}(m)$   & $0$ & $1$ & $2$ & $3$ & $3$ &$4$ & $4$ & $5$ &\ $6$ & $15$ & $15$ & $16$ & $17$ & ...\\ \hline
    
  \end{tabular}
\end{center}

\smallskip
Similarly,  $T(4,9)$ is the link of a singularity with semigroup  generated by $4$ and $9$, whose gap set and gap function are as follows
\[\mathcal{G}_{4,9}=\Z_{<0}\cup\{1,2,3,5,6,7,10,11,14,15,19,23\}\]

\smallskip
\begin{center}  \begin{tabular} {| c || c | c | c | c | c | c |c|c|c|c|c|c|c|c|c|c|}
   \hline
   $m$ & $ \ge 24$ & $23$  & $19$ & $18$ & $16$ &$15$ & $14$ & $11$ & $10$ &$1$& $0$ & $-1$ & $-2$ & ... \\ \hline
    $I_{4,9}(m)$ & $0$ & $1$& $2$ & $2$ & $2$ &$3$ & $4$ & $5$ & $6$& $12$ & $12$ & $13$ & $14$ & ...\\ \hline
    
  \end{tabular}
\end{center}

\end{example}

\medskip

The gap function is quite natural from the point of view of singularity theory.  From the perspective of Heegaard Floer theory, however, it is more natural to consider the shifted gap function, which for an $L$--space knot we define as follows: 
\[J(m)=I(m+g)\]
Motivation for the shift is provided by the the following synthesis of several results of Ozsv\'ath and Szab\'o \cite{os-lspace,OSknots}.

\begin{proposition}[see \expandafter{\cite[Proposition 4.4]{bo-li}}]\label{prop:dandJ}
Let $K$ be an $L$--space knot, and let $q\ge 2g(K)-1$. Then for a particular enumeration of \Spc{} structures $\spinc_m$  by elements $m\in[-q/2,q/2)\cap \Z$, the  $d$--invariants of $q$--surgery on $K$ are given by:
\begin{equation}\label{eq:Jdef} d(S^3_q(K),\mathfrak{s}_m)=\frac{(q-2m)^2-q}{4q}-2J(m).\end{equation}
\end{proposition}

\noindent  While the definition of $J$ above makes sense only for an $L$--space knot, the proposition motivates the following extension to arbitrary knots.

\begin{definition}\label{def:Jgeneral} For a general knot $K$,  define  $J_K(m)$  to be the unique integer making Equation \eqref{eq:Jdef} true. \end{definition}
\begin{remark}With the  above definition, $J(m)$ is easily identified with the function $V_m$ from \cite[Section 2.2]{NiWu}.  
\end{remark}

Turning back to $L$--space knots, the following describes their  $\Upsilon$ functions.
\begin{proposition}[\expandafter{\cite[Theorem 1.15]{OSS}}]\label{prop:UofL}Let $K$ be an $L$--space knot, and $\alpha_i$ as in \eqref{eq:deltaLspace}. Define the sequence $m_k$ inductively by
\begin{align*}
m_0&=0\\
m_{2j}&=m_{2j-1}-1\\
m_{2j+1}&=m_{2j}-2(\alpha_{2j}-\alpha_{2j+1})+1.
\end{align*}
Then
\[\Upsilon_K(t)=\max_{0\le i\le n} m_{2i}-t\alpha_{2i}.\]
\end{proposition}

\noindent Proposition~\ref{prop:UofL} can be reformulated in the following way.
\begin{proposition}\label{cor:upsilon0}
For an $L$--space knot $K$,  the $\Upsilon$ function is given by
\begin{equation}\label{eq:newupsilon}
\Upsilon_K(t)=-\min_{0\le i\le n} t\alpha_{2i}+2J(\alpha_{2i}).
\end{equation}
\end{proposition}
\begin{proof}
It is enough to show that $m_{2i}=-2J(\alpha_{2i})$. To see this we write
\[m_{2i}=2(\alpha_1-\alpha_0)+2(\alpha_3-\alpha_2)+\dots+2(\alpha_{2i-1}-\alpha_{2i-2}).\]

On the other hand, by the definition of $J$, the difference $J(k+1)-J(k)$ is equal to $0$ if $k\in [\alpha_{2j},\alpha_{2j-1})$ for some $j$;
and is equal $-1$ if $k\in[\alpha_{2j+1},\alpha_{2j})$ for some $j$. In particular $J(\alpha_{2j-1})-J(\alpha_{2j})=0$ and
$J(\alpha_{2j})-J(\alpha_{2j+1})=\alpha_{2j+1}-\alpha_{2j}$. Moreover $J(\alpha_0)=0$ by the definition. Therefore an
easy induction yields:
\[-J(\alpha_{2j})=(\alpha_1-\alpha_0)+(\alpha_3-\alpha_2)+\ldots+(\alpha_{2j-1}-\alpha_{2j-2}).\]
\end{proof}

We can rephrase Proposition~\ref{cor:upsilon0} in yet another manner.
Extend $J$ to a piecewise linear function over $\R$. That is, if for $k\in\Z$ we have
$J(k)=J(k+1)$, then set $J|_{[k,k+1]}\equiv J(k)$. If $J(k+1)=J(k)-1$, set for $x\in[0,1]$ $J(k+x)=J(k)-x$.  With this definition, we arrive at Theorem \ref{thm:legendre} from the introduction.
\begin{theorem}\label{thm:Uastransform}
For an $L$--space knot, the $\Upsilon$ function is given by
\begin{equation}\label{eq:fullupsilon}
\Upsilon(t)=\max_{x\in\R} tx-2J(-x).
\end{equation}
Thus $\Upsilon(t)$  is the Legendre transform of the function $x\mapsto 2J(-x)$.
\end{theorem}
\begin{proof}
Notice that 
\[-\min_{x\in\R} tx+2J(x)=\max_{-x\in\R} tx-2J(-x).\] 
Therefore to prove the theorem it suffices to show that,
for a fixed $t$, the minimum of the expression 
\[\tilde{J}(x):=tx+2J(x)\]
is attained at $x=\alpha_{2j}$, for some $j$. We do this
in the following steps. In (a)--(d) we assume that $x$ is an integer.
\begin{itemize}
\item[(a)] Suppose $x\in[\alpha_{2j+1},\alpha_{2j}]$. Then $J(x)=J(\alpha_{2j})+(\alpha_{2j}-x)$; see proof of 
Proposition~\ref{cor:upsilon0}. Thus $\tilde{J}(x)-\tilde{J}(\alpha_{2j})=(2-t)(\alpha_{2j}-x)\ge 0$.
\item[(b)] If $x\in[\alpha_{2j},\alpha_{2j-1}]$, then $J(x)=J(\alpha_{2j})$. It follows that 
$\tilde{J}(x)-\tilde{J}(\alpha_{2j})=t(x-\alpha_{2j})\ge 0$.
\item[(c)] If $x\ge g=\alpha_0$, then $J(x)=0$, hence $\tilde{J}(x)=tx\ge tg=\tilde{J}(g)$.
\item[(d)] If $x\le -g=\alpha_{2n}$, then $J(x)=J(-g)+(-g-x)$, so $\tilde{J}(x)\ge\tilde{J}(-g)$.
\item[(e)] On any interval $[y,y+1]$ where $y\in\Z$, the function $\tilde{J}$ is linear, so it cannot attain its minimum in the interior.
It follows that the minimum of $\tilde{J}$ is attained at an integer point.
\end{itemize}
\end{proof}
\begin{remark}
We notice that the assumption that $t\in[0,2]$ is effectively used in Steps (a) and (b) of the above proof.
\end{remark}

As mentioned in the introduction, it is  natural to wonder  whether Theorem~\ref{thm:Uastransform} holds for other classes of knots, where we define the $J$ function by Equation \eqref{eq:Jdef}. We again stress  that it
cannot hold for every knot since the Legendre transform is a convex function; see Lemma~\ref{lem:isconvex}.

\section{Connected sums of $L$--space knots}

In this section we extend Theorem \ref{thm:legendre} to connected sums of $L$--space knots.   The following result determines the $J$--function in this context.

\begin{proposition}[see \expandafter{\cite[Proposition 5.6]{bo-li}}] Let $K$ be a connected sum of $L$--space knots, $K_1,\ldots,K_n$, and let $J_i$ be the $J$--function of $K_i$.  Then the $J$--function of $K$ (see Definition \ref{def:Jgeneral}) is given by:
\[J=J_1\diamond J_2\diamond\ldots\diamond J_n,\]
where $\diamond$ is the infimum convolution defined in \eqref{eq:diamond}.
\end{proposition}

We have the following generalization of Theorem~\ref{thm:Uastransform}.
\begin{theorem}\label{thm:theoremforsums}
If $K$ is a connected sum of $L$--space knots, then the $\Upsilon$ function for $K$ is the Legendre transform of $x\mapsto 2J(-x)$. 
\end{theorem}
\begin{proof}
According to Lemma~\ref{lem:diamond}, the Legendre transform maps infimum convolutions to sums.  This, together with additivity of $\Upsilon$ under connected sums implies the result.
\end{proof}

\noindent Corollary \ref{cor:independence} follows readily.

\begin{proof}[Proof of Corollary \ref{cor:independence}]  
Let $\alpha$ be a smooth concordance class in the subgroup $\mathcal{L}$ generated by $L$--space knots, so that $\alpha$ can be represented as
$$ \alpha=\sum_{i=1}^n [K_i] - \sum_{l=1}^m [P_l]$$
where each of the $K_i$ and $P_l$ are $L$--space knots.    Suppose that  the $J$ functions for $K=K_1\#...\#K_n$ and $P=P_1\#...\#P_m$ agree.    By Theorem \ref{thm:theoremforsums} we obtain
$$\Upsilon_K=\Upsilon_{P},$$
which implies $\Upsilon_\alpha=\Upsilon_K-\Upsilon_{P}=0$.  Thus, if $\Upsilon_\alpha\ne 0$, then the $J$ functions for $K$ and $P$ are not equal, which shows that $K$ and $P$ are not concordant. Hence $\alpha\ne 0\in\mathcal{L}$.
\end{proof}

\section{Crossing changes}\label{sec:section6}
In this section we establish an inequality for the $J$ functions of knots related by a crossing change. When both knots are $L$--space knots, we recover the crossing change inequality satisfied by their $\Upsilon$ functions (Theorem \ref{thm:ossmain}(c)\,) by taking the Legendre transform and applying Theorem \ref{thm:legendre}.  This implies that the information about Gordian distance  between $L$--space knots contained in $J$ is at least as strong as that coming from $\Upsilon$.  We then show, by way of an example, that  the obstruction from $J$ is strictly better.

\begin{theorem}\label{thm:crossingJ}
Suppose $K_+$ and $K_-$ are arbitrary knots such that $K_-$ can be obtained from $K_+$ by changing a positive crossing.  Then for any $m\in \Z$ 
\begin{equation*} J_{K_+}(m+1) \le J_{K_-}(m)\le J_{K_+}(m)
\end{equation*}
\end{theorem}
\begin{proof}  We focus primarily on the second inequality \[J_{K_-}(m)\le J_{K_+}(m).\]

  Let $S^3_q(K_+)$ be the manifold obtained by $q$--framed surgery on $K_+$ with  $q$ large and odd. Let $W$ denote the $4$--manifold obtained by attaching a $(-1)$--framed $2$-handle to  $S^3_q(K_+)\times [0,1]$ along an unknotted curve in $S^3_q(K_+)\times \{1\}$ which links the crossing strands geometrically two, but  algebraically zero, times (here, when we say ``unknotted", we mean when viewed as a curve in $S^3$).  The oriented boundary of $W$ is $-S^3_{q}(K_+)\cup S^3_q(K_-)$.  One easily calculates $H^2(W)\cong H_2(W,\partial)\cong \Z/q\langle Z\rangle\oplus \Z\langle E\rangle$, where $Z=[\mu\times [0,1]]$ is an oriented meridian of $K_+$ crossed with the interval and $E$ is the cocore of the two-handle.  

For any $m\in[-q/2,q/2]\cap\Z$ we let $\spinct_m$ to be the unique \Spc{} structure on $W$ whose first Chern class is $2mZ+E$ (uniqueness is a consequence of $q$ being odd). We claim it restricts to the \Spc{} structures on $S^3_q(K_0)$ and $S^3_q(K_1)$ denoted 
$\spinc_m$  in the convention of \cite[Section 4]{OSknots}. Indeed,   the \Spc{} structure $\spinc_m$  is defined by the property that it extends over the $2$--handle cobordism from $S^3$ to $S^3_q(K)$ to a \Spc{} structure whose Chern class is $2m-q$ times the Lefschetz dual of the  cocore of the $2$--handle. Since the boundary of the cocore is $\mu_K$,  it follows that the Chern class of $\spinc_m$ is Poincar{\'e} dual to $2m[\mu]\in S^3_q(K_i)$, where $i\in \{+,-\}$. Our claim about $\spinct_m\in $\Spc{$(W)$} follows at once.

We now observe that the rational self--intersection of $2mZ+E$ is $-1$ and does not depend on $m$.  Since $W$ is negative definite, results of  Ozsv\'ath and Szab\'o
(see \cite[Proofs of Theorem 9.1 and Proposition 9.9]{os-absolute}) yield the inequality
\[d(S^3_q(K_-),\spinc_m)-d(S^3_q(K_+),\spinc_m)\ge \frac{c_1^2(\spinct_m)-3\sigma(W)-2\chi(W)}{4}=\frac{-1+3-2}{4}.\]
This inequality, in view of Definition~\ref{def:Jgeneral} translates into $J_{K_+}(m)\ge J_{K_-}(m)$.

An analogous argument establishes the first inequality.  For this consider the  $4$--manifold obtained by attaching a $(-1)$--framed $2$--handle to  $S^3_q(K_-)\times [0,1]$ along an unknot  which links the crossing strands geometrically {\em and}  algebraically twice.  This is a negative definite $4$--manifold with boundary $-S^3_q(K_-)\cup S^3_{q+4}(K_+)$, and we can apply the above inequality.  However, the analysis of the restriction of \Spc{} structures to the boundary is more subtle, and since an alternative proof of the first inequality can be deduced from the proof of  \cite[Theorem 2.14]{bo-li2}, we omit the details here.

\end{proof}

For connected sums of $L$--space knots, it follows that the crossing change obstruction coming from $J$ is at least as strong as that for $\Upsilon$. The following corollary is a restatement
of Theorem~\ref{crossingstrength} from the introduction.

\begin{corollary}\label{needtolabelthistoo} Suppose $K_1$ can be obtained from $K_0$ by changing $p$ positive crossings and $n$ negative crossings.  Then we have the following inequalities: \begin{align}
J_{K_0}(m+p)&\le J_{K_1}(m)\le J_{K_0}(m-n) \label{eq:u2}\\
\Upsilon_{K_0}(t)-nt&\le \Upsilon_{K_1}(t)\le\Upsilon_{K_0}(t)+pt\label{eq:u1}.
\end{align}
If $K_0$ and $K_1$ are connected sums of $L$--space knots, then the second inequalities are determined by the first.
\end{corollary}
\begin{proof}
Both sets of inequalities follow immediately from iterating  the relevant inequalities for a single crossing change, Theorems \ref{thm:crossingJ}  and \ref{thm:ossmain}(c), respectively. 

Suppose now that $K_0$ and $K_1$ are connected sums of $L$--space knots. The inequalities for $J$ imply that for any $m\in\Z$ we have
\[J_{K_0}(-m+p)\le J_{K_1}(-m)\le J_{K_0}(-m-n).\]
Multiply both sides by $2$ and apply the Legendre transform. By Theorem~\ref{thm:theoremforsums}, the Legendre transform of $2J_{K_i}(-m)$ is $\Upsilon_{K_{i}}(t)$.  Recalling that the Legendre transform reverses  inequalities (Lemma \ref{lem:fenchelineq}), together with its behavior under shifts (Equation \eqref{eq:tyh}), the corollary follows immediately.
\end{proof}

The following example (see \cite{bo-li2}) indicates that when analyzing crossing changes between $L$--space knots, the $J$--function is  strictly stronger.  Theorem \ref{crossingstrength} follows at once.

\begin{example}\label{ex:4967}
Let $K_0=T(4,9)$ and $K_1=T(6,7)$, the $(4,9)$ and $(6,7)$ torus knots, respectively. We ask whether three crossing changes can transform $K_0$ into $K_1$. We have
\[\delta(t):=\Upsilon_{T(4,9)}(t)-\Upsilon_{T(6,7)}(t)=
\begin{cases}
3t&t\in[0,\frac13]\\
-3t+2&t\in[\frac13,\frac12]\\
5t-2&t\in[\frac12,\frac23]\\
-t+2&t\in[\frac23,1].
\end{cases}
\]
It is straightforward to compute that $0\le\delta(t)\le 3t$. In particular,  
\[\Upsilon_{T(6,7)}(t)\le\Upsilon_{T(4,9)}(t)\le \Upsilon_{T(6,7)}(t)+3t.\]
Thus $\Upsilon$ (by way of \eqref{eq:u1}) does not obstruct the possibility that  changing three positive crossings of $T(6,7)$ will result in $T(4,9)$.

On the other hand, we can compare $J$ functions.  Referring to the tables in Example \ref{example:gaps} and noting that the Seifert genera of $T(6,7)$ and $T(4,9)$ are  $15$ and $12$, respectively, we see:
\[ J_{T(6,7)}(7):=I_{6,7}(7+15)=3\]
\[J_{T(4,9)}(4):=I_{4,9}(4+12)=2,\]
so that the inequality $J_{T(6,7)}(m+3)\le J_{T(4,9)}(m)$ is violated.  It follows that one cannot change three positive crossings in $T(6,7)$ to obtain $T(4,9)$, and their Gordian distance is therefore at least four.
\end{example}

\

\section{Concluding remarks}

The results from this article indicate that  the information about   $L$--space knots contained in their $d$--invariants is stronger, though perhaps more unwieldy, than that derived from $\Upsilon$.   Of course this might be expected, since the $d$--invariants {\em a priori} determine the knot Floer homology invariants in this context.   Despite this, there is still room to wonder just how tightly the Legendre transform grips the information about $L$--space knots contained in $\Upsilon$.   For instance,  Theorem~\ref{thm:ossmain} implies that if 
$K_0$ and $K_1$ are two knots in $S^3$ admitting a genus $g$ concordance, then for any $t\in[0,1]$ we have
$|\Upsilon_{K_0}(t)-\Upsilon_{K_1}(t)|\le gt$. 
It would be interesting to know whether this result can be obtained using $J$-functions in the case 
$K_0$ and $K_1$ are $L$--space knots.   In this vein, a Fr{\o}shov-type inequality for the $d$--invariants established by Rasmussen seems particularly relevant \cite{Froyshov,RasGoda}. 

Viewing the $\Upsilon$ function of $L$--space knots through the lens of the Legendre transform points to  potential geometric  significance of convexity properties of $\Upsilon$.  Understanding whether such connections exist seems quite important, and we hope to pursue this in future.

\end{document}